\documentclass{amsart}
\usepackage{tikz}
\usepackage{setspace}
\newtheorem{theorem}{Theorem}[section]
\newtheorem{lemma}[theorem]{Lemma}
\theoremstyle{definition}
\newtheorem{definition}[theorem]{Definition}
\newtheorem*{conj}{Conjecture}
\theoremstyle{remark}
\newtheorem{remark}[theorem]{Remark}
\numberwithin{equation}{section}
\newcommand{\Z}{\mathbb{Z}}
\newcommand{\Zn}{(\Z /2\Z)^n}
\newcommand{\Q}{\mathbb{Q}}
\newcommand{\R}{\mathbb{R}}

\newcommand{\OLU}{{\mathcal{O}_L^*}}
\newcommand{\abs}[1]{\lvert#1\rvert}
\newcommand{\ABS}[1]{\left|#1\right|}
\newcommand{\inn}[1]{\langle#1\rangle}
\newcommand{\LOG}{\operatorname{LOG}}
\newcommand{\Gal}{\operatorname{Gal}}
\newcommand{\norm}[2]{\lVert#1\rVert_{#2}}

\providecommand{\MR}{\relax\ifhmode\unskip\space\fi MR }

\allowdisplaybreaks

\begin{document}

\title[A real multi-quadratic Galois case of the B--RV conjecture]
{Bertrand's and Rodriguez Villegas' conjecture for real multi-quadratic Galois extensions of the rationals}

\author{DOHYEONG KIM}
\address{Department of Mathematical Sciences and Institute for Data Innovation in Science, Seoul National University, Gwanak-ro 1, Gwankak-gu, Seoul, South Korea 08826}
\email{dohyeongkim@snu.ac.kr}

\author{SEUNGHO SONG}
\address{Department of Mathematical Sciences, Seoul National University, Gwanak-ro 1, Gwankak-gu, Seoul, South Korea 08826}
\email{shsong0611@snu.ac.kr}

\date{\today}

\begin{abstract}
     The conjecture due to Bertrand and Rodriguez Villegas asserts that the 1-norm of the nonzero element in an exterior power of the units of a number field has a certain lower bound. 
     We prove this conjecture for the exterior square case when the number field is a real multi-quadratic Galois extension of any degree of the rationals.
\end{abstract}

\maketitle


\section{Introduction}

An outstanding conjecture due to Lehmer \cite{Le1933} asserts a lower bound for the height of a unit in a number field.
If one considers the regulator of a number field, a lower bound has been established by Zimmert \cite{Zi1981}.
For exterior powers of units, Bertrand \cite{Be1997} proposed a conjecture, which was confirmed except for the exterior square case, by Amoroso and David \cite{AD1999}.
In 2002, Rodriguez Villegas proposed a sharper version of Bertrand's conjecture, 
also extending Lehmer's conjecture and Zimmert's result,
whose published form appeared relatively recently in \cite{CFS2022}.
We first state the definitions of the necessary terms.
The following definitions are from \cite{CFS2022}.

\begin{definition}
    Let~$L$ be a number field and~$\mathcal{A}_L$ be the set of its Archimedean places. 
    Then the function~$\LOG:\OLU \rightarrow \R^{\mathcal{A}_L}$ is defined by
    \begin{equation*}
        (\LOG(\gamma))_v:=e_v\log\abs{\gamma}_v 
        \text{\quad where \quad}  e_v:=\begin{cases}
        1\ &\text{if}\ v\ \text{is real},\\
        2\ &\text{if}\ v\ \text{is complex}\end{cases}
    \end{equation*}
    for~$\gamma \in \OLU$ and~$v \in \mathcal{A}_L$. 
    Here,~$\abs{\, \cdot \, }_v$ is the absolute value associated to~$v$ extending the absolute value on~$\Q$.
\end{definition}

Consider the orthonormal basis ~$\{\delta^v\}_{v \in \mathcal{A}_L}$ of~$\R^{\mathcal{A}_L}$ given by 
\begin{equation*}
    \delta_{w}^{v}:=\begin{cases}
    1\ &\text{if}\ \ w=v,\\
    0\ &\text{if}\ \ w\neq v \end{cases}
\end{equation*}
for~$w\in \mathcal{A}_L$.
Let~$\mathcal{A}_L^{[j]}$ be the set of subsets of~$\mathcal{A}_L$ with cardinality~$j$.
For each ~$I\in\mathcal{A}_L^{[j]}$, fix an ordering ~$\{v_1,v_2,...,v_j\}$ of elements of~$I$.
Define
\begin{equation*}
    \delta^{I}:=\delta^{v_1}\wedge \delta^{v_2} 
    \wedge \cdots \wedge \delta^{v_j}
\end{equation*}
to get an orthonormal basis ~$\{\delta^{I}\}_{I\in\mathcal{A}_L^{[j]}}$ of ~$\bigwedge^j\R^{\mathcal{A}_L}$.
For~$w=\sum\limits_{I \in \mathcal{A}_L^{[j]}} c_I\delta^I \in \bigwedge^j\R^{\mathcal{A}_L}$, define its 1-norm as
\begin{equation*}
    \norm{w}{1}:=\sum\limits_{I \in \mathcal{A}_L^{[j]}}\abs{c_I}.
\end{equation*}

Now we state the Bertrand's and Rodriguez Villegas' conjecture from \cite{CFS2022}.
\begin{conj}[Bertrand--Rodriguez Villegas]
    There exist two absolute constants \\~$c_0>0$ and ~$c_1>1$ such that for any number field~$L$ and any~$j \in \Z_{>0}$, the following inequality 
    \begin{equation} \label{brvin}
        \norm{w}{1} \geq c_0c_1^j
    \end{equation}
    holds for any nonzero~$w \in \bigwedge^j\LOG(\OLU) \subset \bigwedge^j\R^{\mathcal{A}_L}$.
\end{conj}

For some values of~$j$, the conjecture reduces to the case when~$w$ is a pure wedge product, 
i.e.~$w=\LOG(\epsilon_1)\wedge \cdots \wedge \LOG(\epsilon_j)$ where~$\epsilon_i$ are multiplicatively independent units of~$L$.
This includes the case~$j=1$, where the conjecture is equivalent to Lehmer's conjecture \cite{Le1933}, 
and the case~$j=\operatorname{rank}_{\Z}(\OLU)$, where the conjecture is equivalent to Zimmert's theorem on regulators \cite{Zi1981}.
The case~$j=\operatorname{rank}_{\Z}(\OLU)-1$ also reduces to the case where~$w$ is a pure wedge product.
This is because every element of~$\bigwedge^j\LOG(\OLU)$ can be written in the form~$d \cdot \LOG(\epsilon_1)\wedge \cdots \wedge \LOG(\epsilon_j)$ where~$d \in \Z$ and~$\LOG(\epsilon_1),...,\LOG(\epsilon_{j+1})$ are the basis of~$\LOG(\OLU)$
as shown in the Lemma 28 of \cite{CFS2022}. 
In general, showing the inequality \eqref{brvin} for nonzero pure wedge products does not guarantee that the inequality holds for all nonzero elements of~$\bigwedge^j\LOG(\OLU)$.
However, previous works on the Bertrand--Rodriguez Villegas conjecture only involves pure wedge products as stated in \cite{CFS2022}. 

For a totally real number field~$L$, Costa and Friedman \cite{CF1991} showed that for independent elements~$\epsilon_1,...,\epsilon_j$ of~$\OLU$, the inequality
\begin{equation} \label{cfineq}
    \norm{\LOG(\epsilon_1)\wedge \cdots \wedge \LOG(\epsilon_j)}{2}
    >\frac{1}{(j+2)\sqrt{j}}\left(\frac{[L:\Q]}{j}\right)^{j/2} 1.406^j
\end{equation}
holds for~$1\leq j<[L:\Q]$. Therefore, with some calculations it follows that 
\begin{equation*}
    \norm{\LOG(\epsilon_1)\wedge \cdots \wedge \LOG(\epsilon_j)}{1}
    >\frac{1}{(j+2)\sqrt{j}}\left(\frac{[L:\Q]}{j}\right)^{j/2} 1.406^j\geq 0.001 \cdot 1.4^j
\end{equation*}
also holds, proving the conjecture for pure wedge products of units of totally real fields \cite{CFS2022}. 
The inequality \eqref{cfineq} is stronger than that of Bertrand-Rodriguez Villegas conjecture, in the sense that the 2-norm is used and that the lower bound diverges as the degree~$[L:\Q]$ goes to infinity for a fixed~$j$. 
This implies that the important part of the conjecture is that the inequality \eqref{brvin} holds for all nonzero~$w\in\bigwedge^j\LOG(\OLU)$, even for~$w$ that is not a pure wedge product. 

In this paper, we give a lower bound of~$\norm{w}{1}$ for nonzero~$w \in \bigwedge^2\LOG(\OLU)$ when ~$L$ is a real Galois extension of~$\Q$ with~$\Gal(L/\Q) \cong \Zn$ for some integer~$n \geq 2$.
When~$n\geq3$, this case does not necessarily reduce to the case when~$w$ is a pure wedge product.
Our main theorem is as follows. 
\begin{theorem} \label{multi}
    Let~$n \geq 2$ be a positive integer, and~$L$ be a real Galois extension of~$\Q$ with~$\Gal(L/\Q) \cong \Zn$.
    Then, 
    \begin{equation*}
        \norm{w}{1} \geq 2 \log\left(\frac{1+\sqrt{5}}{2}\right)\log(1+\sqrt{2}) \approx 0.8483
    \end{equation*}
    for any nonzero~$w \in \bigwedge^2 \LOG (\OLU).$
\end{theorem}
While our result is limited to the exterior square case where~$L$ belongs to a specific family of number fields,
we provide a constant lower bound of~$\norm{w}{1}$ independent of the degree~$[L:\Q]=2^n$.

The key idea of the proof of the main theorem is to use the fundamental units of the quadratic subfields of~$L$.
The field $L$ has exactly~$2^n-1$ quadratic subfields, and the fundamental unit of these quadratic subfields generate a subgroup~$E$ of~$\OLU$. 
Since the Galois module structure of $E$ is known,~$\LOG(E)$ is easier to handle than ~$\LOG(\OLU)$.
We first give a lower bound of ~$\norm{w}{1}$ for nonzero element ~$w$ of~$\bigwedge^2\LOG(E)$ using elementary linear algebra, and extend this result to ~$\bigwedge^2\LOG(\OLU)$.
The term ~$\log(\frac{1+\sqrt{5}}{2})$ and ~$\log(1+\sqrt{2})$ from the lower bound comes from the fact that the fundamental unit~$u>1$ of a quadratic field~$\Q(\sqrt{m})$ satisfy~$u\geq\frac{1+\sqrt{5}}{2}$, and when~$m\neq5$,~$u\geq1+\sqrt{2}$.

The proof of the main theorem is given in the following sections.
In section ~\ref{secunits} we show some properties of subgroups of~$\Zn$ of index $2$. 
We introduce a subgroup~$E$ of~$\OLU$ generated by positive units of quadratic subfields of~$L$,
and prove that for any unit~$u\in\OLU$,~$u^{2^{n-1}}\in E$.
In section ~\ref{secnorm}, 
using the values of~$\LOG$ of the fundamental units, 
we give a basis of~$\bigwedge^2\LOG(E)$ and the 1-norm of its elements.
In section ~\ref{seclower} we give a lower bound for the elements of~$\bigwedge^2\LOG(E)$ using elementary linear algebra, and use this result to prove our main theorem.

\subsection*{Acknowledgement}
The work of S.S. was supported by College of Natural Sciences Undergraduate Internship Program from Seoul National University.
D.K. is supported by the National Research Foundation of Korea (NRF)\footnote{the grants No.\,RS-2023-00301976 and No.\,2020R1C1C1A0100681913}.


\section{The Structure of Units} \label{secunits}

Fix a positive integer~$n\geq2$ and consider a totally real number field~$L$ such that the group ~$\Gal(L/\Q)$ is isomorphic to~$\Zn$. 
We will identify~$\Gal(L/\Q)$ and~$\Zn$ via a fixed isomorphism~$\lambda\colon \Zn\xrightarrow{\sim}\Gal(L/\Q)$.
We also fix an embedding of~$L$ into~$\R$,
so the elements of~$\Gal(L/\Q)$ are in one-to-one correspondence with the places of~$L$.
Based on the choices we made, we identify~$\mathcal{A}_L$ with~$\Gal(L/\Q)$ throughout.
Let~$A:=\Zn \backslash \{0\}$ and let~$B$ be the set of subgroups of~$\Zn$ of index~$2$. 
Then each element of~$B$ corresponds to a quadratic subfield of~$L$. 
We first define some notions that will be used throughout this paper.

\begin{definition}
    We define~$\inn{\, ,\, }:\Zn \times \Zn \rightarrow \Z/2\Z$ as follows. 
    For two elements~$p=(p_1,p_2,...,p_n)$ and~$q=(q_1,...,q_n)$ of~$\Zn$ with~$p_i,q_i \in \Z/2\Z~$ for~$1\leq i \leq n$, define
    \begin{equation*}
        \inn{p,q}:=\sum\limits_{i=1}^{n}p_iq_i,
    \end{equation*}
    where the addition and multiplication are taking place in the field~$\Z/2\Z$.
    For~$p\in A$, define
    \begin{equation*}
        p^\perp:=\{z \in \Zn \ | \ \inn{p,z}=0 \}.
    \end{equation*}    
\end{definition}

For any~$p\in A$,~$p^\perp$ is the kernel of the surjective group homomorphism that sends~$z\in \Zn$ to~$\inn{p,z}$. 
Thus~$p^\perp \in B$. In fact, every element of~$B$ are of this form, as stated in the following lemma.

\begin{lemma} \label{subg}
    Let~$A,B$ be defined as above.
    \begin{enumerate}
        \item There is a bijective map~$\phi$ between~$A$ and~$B$ given by
        \begin{align*}
            \phi: \ A \  &\rightarrow \ B \\
            x \ &\mapsto \ x^\perp            
        \end{align*}
        and, in particular,~$\abs{B}=2^n-1$. 
        \item For any element~$x$ of~$A$,~$x$ is in exactly~$2^{n-1}-1$ elements of~$B$.
        \item For any two distinct elements~$x$ and~$y$ of~$A$,~$\abs{x^\perp \cap y^\perp}=2^{n-2}$.
    \end{enumerate}
\end{lemma}
\begin{proof}
    (1) First we show~$\phi$ is surjective. 
    Take any~$b \in B$. 
    Since~$\Zn$ is abelian, ~$b$ is a normal subgroup of~$\Zn$, and~$\Zn/b \cong \Z/2\Z$. 
    Thus there exists a surjective homomorphism~$h_b: \Zn \rightarrow \Z/2\Z$ such that its kernel is~$b$. 
    Let~$e_i$ be a~$i$-th standard basis of~$\Zn$, i.e. the element with~$i$-th component~$1$ and other components~$0$. 
    If we let~$a=\sum\limits_{i=1}^{n} h_b(e_i)e_i$, then~$a\in A$ and~$h_b$ sends~$z \in \Zn$ to~$\inn{a,z}$.
    Thus~$\phi(a)=\ker(h_b)=b$, and~$\phi$ is surjective. \\
    We now show~$\phi$ is injective. 
    Let~$x,y$ be distinct elements of~$A$. 
    Since~$x\neq y$, there exists~$1\leq j\leq n$ such that~$\inn{x,e_j}\neq \inn{y,e_j}$. 
    If~$e_j \in x^\perp$, then~$e_j \notin y^\perp$ and if~$e_j \notin x^\perp$, then~$e_j \in y^\perp$. Thus~$x^\perp \neq y^\perp$ and~$\phi$ is injective. Therefore~$\phi$ is bijective and~$\abs{B}=\abs{A}=2^n-1$.\\
    (2) By (1), we only need to count the number of elements of~$A$ which are perpendicular to~$x$. Let the~$j$-th component of~$x$ be 1. 
    Then for any~$z \in \Zn$, only one of~$z$ and~$z+e_j$ is perpendicular to~$x$. 
    Thus ~$\frac{1}{2}\abs{\Zn}=2^{n-1}$ elements of~$\Zn$ are perpendicular to~$x$. Excluding 0, exactly ~$2^{n-1}-1$ elements of~$A$ are perpendicular to~$x$. \\
    (3) Since~$x\neq y$, there exists~$1\leq j\leq n$ such that~$\inn{x,e_j}\neq \inn{y,e_j}$. 
    Without loss of generality, let~$\inn{x,e_j}=0$ and~$\inn{y,e_j}=1$. 
    Then for any~$z \in x^\perp$, since~$\inn{z+e_j,x}=0$, ~$z+e_j\in x^\perp$. Since~$\inn{z,y} \neq \inn{z+e_j,y}$, only one of~$z$ and~$z+e_j$ is also in~$y^\perp$. 
    Thus exactly~$\frac{1}{2}\abs{x^\perp}=2^{n-2}$ elements of~$x^\perp$ are in~$x^\perp \cap y^\perp$.
\end{proof}

For each~$a \in A$,~$a^\perp$ is a subgroup of~$\Zn$ of index~$2$, so~$\lambda(a^\perp)$ is a subgroup of~$\Gal(L/\Q)$ of index~$2$. 
Let~$\Q(\sqrt{d_a})$ be the quadratic subfield of~$L$ that corresponds to~$\lambda(a^\perp)$, 
and let~$u_a>1$ be its fundamental unit.
Then every unit of~$\Q(\sqrt{d_a})$ is of the form~$\pm u_a^m$ where~$m\in\Z$.
Let~$E:=\{ \prod\limits_{a \in A}u_a^{m_a}\  |\  m_a \in \Z\}$.

\begin{lemma} \label{fini}
    For any unit ~$u$ of ~$L$, ~$u^{2^{n-1}}$ is in ~$E$.
\end{lemma}
\begin{proof}
    For any~$a \in A$,~$N_{L/\Q(\sqrt{d_a})}(u)$, its norm in $\Q(\sqrt{d_a})$, is a product of conjugates of~$u$, so it is a unit of~$L$.
    Since~$N_{L/\Q(\sqrt{d_a})}(u)\in \Q(\sqrt{d_a})$, it is also a unit of~$\Q(\sqrt{d_a})$.
    Hence
    \begin{equation*}
        \prod_{\sigma \in \Gal(L/\Q(\sqrt{d_a}))}\sigma(u) = \pm u_a^{m_a}
    \end{equation*}
    for some~$m_a \in \Z$.
    By multiplying above equations for every ~$a$ and taking its absolute value, we get
    \begin{align*}
        \prod_{a \in A}u_a^{m_a} &= \ABS{\prod_{a \in A}\prod_{\sigma \in \Gal(L/\Q(\sqrt{d_a}))}\sigma(u)}\\
        &=\ABS{ \prod_{b \in B}\prod_{x \in b} \lambda(x)(u) }
    \end{align*}
    where the second equality comes from the bijection between~$A$ and~$B$ in (1) of Lemma ~\ref{subg}.
    In the product ~$\prod\limits_{b \in B}\prod\limits_{x \in b}$, the element ~$0 \in \Zn$ appears ~$\abs{B}=2^n-1$ times, 
    while other elements appear~$2^{n-1}-1$ times by (2) of Lemma \ref{subg}.
    Thus in the product~$\prod\limits_{a \in A}\prod\limits_{\sigma \in \Gal(L/\Q(\sqrt{d_a}))}$,
   ~$\sigma$ is the identity element of~$\Gal(L/\Q)$~$2^n-1$ times, while it is equal to any given non-trivial element~$2^{n-1}-1$ times. 
   Thus we have
    \begin{align*}
        \prod_{a \in A}u_a^{m_a} & = \ABS{u^{2^n-1} \cdot\prod_{\sigma \in \Gal(L/\Q)\backslash\{id\}}\sigma(u)^{2^{n-1}-1}} \\
        &=\ABS{u^{2^{n-1}} \cdot \prod_{\sigma \in \Gal(L/\Q)}\sigma(u)^{2^{n-1}-1} } \\
        &=\abs{u^{2^{n-1}}\cdot N_{L/\Q}(u)^{2^{n-1}-1}}.
    \end{align*}
    Since ~$u\in\OLU$, we have ~$\abs{N_{L/\Q}(u)}=1$ and therefore ~$u^{2^{n-1}}= \prod\limits_{a \in A}u_a^{m_a}\in E$.
\end{proof}


\section{The 1-norm of the Wedge Products} \label{secnorm}

Since ~$E$ is generated by $\{u_a\}_{a\in A}$,~$\LOG(E)$ is generated by~$\{\LOG(u_a)\}_{a\in A}$.
Identifying~$\Gal(L/\Q)$ with~$\mathcal{A}_L$, the basis of~$\R^{\mathcal{A}_L}$ is
$\{\delta^{\sigma}\}_{\sigma \in \Gal(L/\Q)} =\{\delta^{\lambda(x)}\}_{x \in \Zn}$. 
We now compute~$\LOG(u_a)$ for~$a \in A$. 
If~$\lambda(x) \in \Gal(L/\Q(\sqrt{d_a}))$, it fixes~$u_a$ and if~$\lambda(x) \notin \Gal(L/\Q(\sqrt{d_a}))$, it sends~$u_a$ to its conjugate,~$\pm \frac{1}{u_a}$. Also,~$\lambda(x) \in \Gal(L/\Q(\sqrt{d_a}))$ if and only if~$x \in a^\perp$, which is equivalent to~$\inn{a,x}=0$.
Thus, the coefficient of the basis~$\delta^{\lambda(x)}$ of~$\LOG(u_a)$ is
\begin{equation*}
    (\LOG(u_a))_{\lambda(x)}=(-1)^{\inn{a,x}}\log(u_a).
\end{equation*}
Here, for~$n \in \Z/2\Z$,~$(-1)^n=1$ if~$n=0$ and~$(-1)^n=-1$ if~$n=1$.
Now give any ordering to~$\Zn$ such that~$0$ is the smallest element, and also give this ordering to~$A$.
For~$b,c \in A$ and~$x,y \in \Zn$ with~$b<c$ and~$x<y$, the coefficient of the basis~$\delta^{\lambda(x)} \wedge \delta^{\lambda(y)}$ of~$\LOG(u_b) \wedge \LOG(u_c)$ is
\begin{equation*}
    (\LOG(u_b) \wedge \LOG(u_c))_{(\lambda(x),\lambda(y))}=\log(u_b)\log(u_c)\{(-1)^{\inn{b,x}+\inn{c,y}}-(-1)^{\inn{b,y}+\inn{c,x}}\}.
\end{equation*}

The abelian group ~$\bigwedge^2\LOG(E)$ is generated by~$\LOG(u_b) \wedge \LOG(u_c)$'s where~$b,c \in A$ and~$b<c$.
Thus ~$w \in \bigwedge^2\LOG(E)$ can be written as
\begin{equation} \label{wexp}
    w=\sum_{\substack{b,c \in A \\ b<c}}n_{b,c}\LOG(u_b) \wedge \LOG(u_c)
\end{equation}
where all~$n_{b,c}$'s are integers. 
Then its coefficient of the basis~$\delta^{\lambda(x)} \wedge \delta^{\lambda(y)} (x<y)$ is 
\begin{equation} \label{wcoeff}
    (w)_{(\lambda(x),\lambda(y))}=\sum_{\substack{b,c \in A \\ b<c}}n_{b,c}\log(u_b)\log(u_c)\{(-1)^{\inn{b,x}+\inn{c,y}}-(-1)^{\inn{b,y}+\inn{c,x}}\}.
\end{equation}
If we let~$b+c=d$ and~$x+y=z$, then 
\begin{align*}
    (-1)^{\inn{b,x}+\inn{c,y}}-(-1)^{\inn{b,y}+\inn{c,x}}
    &=(-1)^{\inn{b,y}+\inn{c,x}}\{(-1)^{\inn{b+c,x+y}}-1\} \\
    &=(-1)^{\inn{b,z-x}+\inn{d-b,x}}\{(-1)^{\inn{d,z}}-1\} \\
    &=(-1)^{\inn{b,z}+\inn{d,x}+2\inn{b,x}}\{(-1)^{\inn{d,z}}-1\} \\
    &=(-1)^{\inn{b,z}+\inn{d,x}}\{(-1)^{\inn{d,z}}-1\}.
\end{align*}
Then \eqref{wcoeff} and the above equation imply that the 1-norm of $w$ is 
\begin{align*} 
    &\sum_{\substack{x,y \in \Zn \\ x<y}} \ABS{\sum_{\substack{b,c \in A \\ b<c}}n_{b,c}\log(u_b)\log(u_c)\{(-1)^{\inn{b,x}+\inn{c,y}}-(-1)^{\inn{b,y}+\inn{c,x}}\}}  \\
    =&\sum\limits_{z \in A}\sum_{\substack{x \in \Zn \\ x<z+x}} \Bigg| \sum\limits_{d \in A}\sum_{\substack{b \in A \\ b<b+d}} n_{b,d+b}\log(u_b)\log(u_{d+b})
       \cdot(-1)^{\inn{b,z}+\inn{d,x}}\{(-1)^{\inn{d,z}}-1\}\Bigg|  \\
    =&\sum\limits_{z \in A}\sum_{\substack{x \in \Zn \\ x<z+x}} \Bigg| \sum_{\substack{d \in A \\ \inn{d,z}=1}}\sum_{\substack{b \in A \\ b<b+d}}n_{b,d+b}\log(u_b)\log(u_{d+b})(-1)^{\inn{b,z}+\inn{d,x}}(-2)\Bigg|  \\
    =&\sum\limits_{z \in A}\sum_{\substack{x \in \Zn \\ x<z+x}} \Bigg| \sum_{\substack{d \in A \\ \inn{d,z}=1}}\sum_{\substack{b \in A \\ b<b+d}}2n_{b,d+b}\log(u_b)\log(u_{d+b})(-1)^{\inn{b,z}+\inn{d,x}}\Bigg|. 
\end{align*}
In summary, 
\begin{equation} \label{wnorm}
    \norm{w}{1}=\sum\limits_{z \in A}\sum_{\substack{x \in \Zn \\ x<z+x}} \Bigg| \sum_{\substack{d \in A \\ \inn{d,z}=1}}\sum_{\substack{b \in A \\ b<b+d}}2n_{b,d+b}\log(u_b)\log(u_{d+b})(-1)^{\inn{b,z}+\inn{d,x}}\Bigg| 
\end{equation}
for ~$w$ given by ~\eqref{wexp}.

\section{The Lower Bound of the 1-norm} \label{seclower}

To give a lower bound for \eqref{wnorm}, we turn to the following lemma. 
\begin{lemma} \label{lin}
    Let~$m$ be a nonnegative integer, and let~$P$ be a~$2^m \times 2^m$ matrix with each entries~$1$ or~$-1$. Assume that the rows of~$P$ are perpendicular to each other. Then for any vector~$X \in \R^{2^m}$, the following inequality
    \begin{equation*}
        \norm{PX}{1} \geq 2^m \norm{X}{\infty}
    \end{equation*}
    holds. 
    Here,~$\norm{\,\cdot\,}{1}$ and~$\norm{\,\cdot\,}{\infty}$ are the usual~$1$-norm and~$\infty$-norm defined in~$\R^{2^{m}}$.
\end{lemma}
\begin{proof}
    By the assumptions,~$PP^T=2^mI$. In particular,~$P^T$ is invertible. Let~$X=P^TY$.
    Note that
    \begin{equation} \label{peq}
        \norm{PX}{1}=\norm{PP^TY}{1}=\norm{2^mY}{1}=2^m\norm{Y}{1}.
    \end{equation}
    For~$1 \leq i \leq 2^m$, let the~$i$-th component of~$X, Y$ be~$x_i,y_i$ respectively. 
    For~$1 \leq i,j \leq 2^m$, let~$P_{i,j}$ be the~$i$-th row,~$j$-th column entry of~$P$. 
    Then from~$X=P^TY$, 
    \begin{equation*}
        x_i=\sum\limits_{j=1}^{2^m} P_{j,i} y_j
    \end{equation*}
    for any~$1 \leq i \leq 2^m$ and thus
    \begin{equation*}
        \abs{x_i}=\ABS{\sum\limits_{j=1}^{2^m} P_{j,i} y_j} \leq \sum\limits_{j=1}^{2^m} \abs{P_{j,i}}\abs{y_j}=\sum\limits_{j=1}^{2^m} \abs{y_j}=\norm{Y}{1}.
    \end{equation*}
    It follows that
    \begin{equation*}
        \norm{X}{\infty} =\max_{1\leq i \leq m} \{\abs{x_i}\}\leq \norm{Y}{1}
    \end{equation*}
    and with \eqref{peq}, we are done.
\end{proof}

Fix~$z \in A$. 
Let~$X_z:=\{x \in \Zn\  | \ x<z+x \}$ and~$D_z:=\{d \in A \ | \ \inn{d,z}=1 \}$. 
Then~$\abs{X_z}=2^{n-1}$ and by (2) of Lemma \ref{subg}, we also have ~$\abs{D_z}=2^{n-1}$. Let~$X_z=\{x_1,x_2,...,x_{2^{n-1}}\}$ and~$D_z=\{d_1,d_2,...,d_{2^{n-1}}\}$. 
The order of the elements is irrelevant in the subsequent arguments.
Let~$V$ be a vector in~$\R^{2^{n-1}}$, whose~$i$-th component is
\begin{equation} \label{viv}
    v_i:=\sum_{\substack{b \in A \\ b<b+d_i}}2n_{b,d_i+b}\log(u_b)\log(u_{d_i+b})(-1)^{\inn{b,z}}
\end{equation}
for~$1\leq i \leq 2^{n-1}$.
Let~$P_z$ be a~$2^{n-1}\times2^{n-1}$ matrix whose entry of~$i$-th row,~$j$-th column is~$(-1)^{\inn{x_i,d_j}}$ for~$1\leq i,j \leq 2^{n-1}$. 
Then each entry of~$P_z$ is either~$1$ or~$-1$. 
Now we show that the rows of~$P_z$ are perpendicular. 
Let~$1\leq i<j\leq 2^{n-1}$.
Then the inner product in~$\R^{2^{n-1}}$ of~$i$-th row and~$j$-th row of~$P_z$ is
\begin{equation*}
    \sum\limits_{k=1}^{2^{n-1}}(-1)^{\inn{x_i,d_k}} \cdot (-1)^{\inn{x_j,d_k}}
    =\sum\limits_{k=1}^{2^{n-1}}(-1)^{\inn{x_i+x_j,d_k}} 
    =\sum\limits_{d \in D_z}(-1)^{\inn{x_i+x_j,d}}.
\end{equation*}
Take any element~$d'$ of~$D_z$. Then~$D_z=d'+z^\perp$, and thus
\begin{equation*}
    \sum\limits_{d \in D_z}(-1)^{\inn{x_i+x_j,d}}=\sum\limits_{d \in z^\perp}(-1)^{\inn{x_i+x_j,d+d'}}=(-1)^{\inn{x_i+x_j,d'}}\sum\limits_{d \in z^\perp}(-1)^{\inn{x_i+x_j,d}}.
\end{equation*}
Since~$x_i$ and~$x_j$ are elements of~$X_z$, we have~$x_i<x_i+z$ and~$x_j<x_j+z$.
Therefore~$x_i+x_j \neq z$; otherwise, substituting $z=x_i+x_j$ to the two order relations would lead to $x_i<x_j$ and $x_j < x_i$, a contradiction.
Also,~$x_i+x_j\neq 0$ since they are distinct. 
By Lemma \ref{subg} (3),
\begin{align*}
    \sum\limits_{d \in z^\perp}(-1)^{\inn{x_i+x_j,d}}
    &=\ABS{\left\{d\in z^\perp \middle| \inn{x_i+x_j,d}=0 \right\}}-\ABS{\left\{d\in z^\perp \middle| \inn{x_i+x_j,d}=1 \right\}} \\
    &=\ABS{z^\perp \cap (x_i+x_j)^\perp}-(\ABS{z^\perp}-\ABS{z^\perp \cap (x_i+x_j)^\perp})\\
    &=2\ABS{z^\perp \cap (x_i+x_j)^\perp}-\ABS{z^\perp}\\
    &=2\cdot 2^{n-2}-2^{n-1}\\
    &=0.
\end{align*}
Therefore the~$i$-th row and~$j$-th row of~$P_z$ are perpendicular.
Now by Lemma \ref{lin}, 
\begin{equation*}
    \sum\limits_{i=1}^{2^{n-1}}\ABS{\sum\limits_{j=1}^{2^{n-1}}v_j(-1)^{\inn{x_i,d_j}}}=\norm{P_zV}{1}\geq 2^{n-1} \norm{V}{\infty}=2^{n-1}\max\limits_{1\leq j\leq n}\abs{v_j}.
\end{equation*}
Rewriting the above inequality without using the indices $i$ and $j$, and substituting $v_j$ with ~\eqref{viv}, we have
\begin{align*}
  & \sum_{\substack{x \in \Zn \\ x<z+x}} \Bigg| \sum_{\substack{d \in A \\ \inn{d,z}=1}}\sum_{\substack{b \in A \\ b<b+d}}2n_{b,d+b}\log(u_b)\log(u_{d+b})(-1)^{\inn{b,z}}(-1)^{\inn{d,x}}\Bigg|\\
  \geq {} & 2^{n-1}\max_{\substack{d \in A \\ \inn{d,z}=1}}\left\{\ABS{\sum_{\substack{b \in A \\ b<b+d}}2n_{b,d+b}\log(u_b)\log(u_{d+b})(-1)^{\inn{b,z}}} \right\}.
\end{align*}
Applying the above inequality to each summand of \eqref{wnorm}, we get
\begin{align*} 
    \norm{w}{1}
    \geq 2^n \sum\limits_{z \in A}\max_{\substack{d \in A \\ \inn{d,z}=1}}\left\{\ABS{\sum_{\substack{b \in A \\ b<b+d}}n_{b,d+b}\log(u_b)\log(u_{d+b})(-1)^{\inn{b,z}}} \right\}
\end{align*}
and for any ~$d'\in A$, 
\begin{align*} 
    \norm{w}{1}
    &\geq 2^n  \sum_{\substack{z \in A \\ \inn{d',z}=1}}\max_{\substack{d \in A \\ \inn{d,z}=1}}\left\{\ABS{\sum_{\substack{b \in A \\ b<b+d}}n_{b,d+b}\log(u_b)\log(u_{d+b})(-1)^{\inn{b,z}}} \right\} \\
    &\geq 2^n  \sum_{\substack{z \in A \\ \inn{d',z}=1}}\ABS{\sum_{\substack{b \in A \\ b<b+d'}}n_{b,d'+b}\log(u_b)\log(u_{d'+b})(-1)^{\inn{b,z}}}
\end{align*}
holds. Therefore we have
\begin{align} \label{fstlm}
    \norm{w}{1}
    \geq 2^n \max_{\substack{d \in A }}\left\{ \sum_{\substack{z \in A \\ \inn{d,z}=1}}\ABS{\sum_{\substack{b \in A \\ b<b+d}}n_{b,d+b}\log(u_b)\log(u_{d+b})(-1)^{\inn{b,z}}}\right\}.
\end{align}

This time, fix~$d\in A$.
The following arguments are similar to above. 
Let~$B_d:=\{b \in \Zn\  | \ b<d+b \}$ and~$Z_d:=\{z \in A \ | \ \inn{d,z}=1 \}$. Then~$\abs{B_d}=\abs{Z_d}=2^{n-1}$. 
Let~$B_d=\{b_1,b_2,...,b_{2^{n-1}}\}$ with~$b_1=0$ and let~$Z_d=\{z_1,...,z_{2^{n-1}}\}$.
Let~$V'$ be a vector in~$\R^{2^{n-1}}$, whose~$1$st element is~$v_1'=0$ and~$i$-th component is
\begin{equation*} 
    v_i':=n_{b_i,d+b_i}\log(u_b)\log(u_{d+b_i})
\end{equation*}
for ~$1 < i \leq 2^{n-1}$.
Let~$Q_d$ be a~$2^{n-1}\times2^{n-1}$ matrix whose entry of~$i$-th row,~$j$-th column is~$(-1)^{\inn{z_i,b_j}}$ for~$1\leq i,j \leq 2^{n-1}$. 
Then each entry of~$Q_d$ is either~$1$ or~$-1$. 
We now show that the rows of~$Q_d$ are perpendicular. 
The previous argument on the rows of~$P_z$ shows that in fact the columns of~$Q_d$ are perpendicular.
Hence~$Q_d^TQ_d=2^{n-1}I=Q_dQ_d^T$, and the rows of~$Q_d$ are also perpendicular.
Then by Lemma \ref{lin},
\begin{equation} \label{qeq}
    \norm{Q_dV'}{1} \geq 2^{n-1} \norm{V'}{\infty}.
\end{equation}
The left hand side of \eqref{qeq} is
\begin{equation*} 
    \sum\limits_{i=1}^{2^{n-1}}\ABS{\sum\limits_{j=1}^{2^{n-1}}v_j'(-1)^{\inn{b_j,z_i}}}
    =\sum_{\substack{z \in A \\ \inn{d,z}=1}}\ABS{\sum_{\substack{b \in A \\ b<b+d}}n_{b,d+b}\log(u_b)\log(u_{d+b})(-1)^{\inn{b,z}}}
\end{equation*}
while the right hand side of \eqref{qeq} is
\begin{equation*} 
    2^{n-1}\max\{0,\abs{v_2'},...,\abs{v_{2^{n-1}}'}\}=
    2^{n-1}\max_{\substack{b \in A \\ b<b+d}}\{ \abs{n_{b,d+b}\log(u_b)\log(u_{d+b})}\}.
\end{equation*}
Applying \eqref{qeq} to the right hand side of \eqref{fstlm}, we get
\begin{equation*}
    \norm{w}{1} \geq 2^{2n-1}\max_{\substack{d \in A}}\left\{\max_{\substack{b \in A \\ b<b+d}}\{ \abs{n_{b,d+b}\log(u_b)\log(u_{d+b})}\}\right\}.
\end{equation*}
By Lemma \ref{quadunit} below, if~$b, c\in A$ and~$b<c$,~$\log(u_b)\log(u_c)\geq \log(\frac{1+\sqrt{5}}{2})\log(1+\sqrt{2})$. 
If there exists~$b,c \in A$ such that~$b<c$ and~$n_{b,c}\neq 0$, we have 
\begin{equation} \label{proto}
    \norm{w}{1} \geq 2^{2n-1}\log\left(\frac{1+\sqrt{5}}{2}\right)\log(1+\sqrt{2}).
\end{equation}

\begin{lemma} \label{quadunit}
    For a square-free integer~$m$, let~$v_m>1$ be the fundamental unit of~$\Q(\sqrt{m})$. 
    Then,~$v_m \geq \frac{1+\sqrt{5}}{2}$ and the equality holds if and only if~$m=5$. 
    Furthermore, if~$m \neq 5$,~$v_m \geq 1+\sqrt{2}$ and the equality holds if and only if~$m=2$. 
\end{lemma}
\begin{proof}
    First, consider the case ~$m \equiv 2,3\pmod 4$. 
    The fundamental unit~$v_m$ is of the form ~$a+b\sqrt{m}$ where
   ~$a,b$ are the smallest positive integers satisfying~$a^2 - mb^2=\pm 1$.
    If~$m=2$,~$v_2=1+\sqrt{2}$   and if~$m \neq 2$,~$v_m \geq 1+\sqrt{m} > 1+\sqrt{2}$. \par
    Now consider the case~$m \equiv 1 \pmod{4}$.
    Then ~$v_m$ is of the form~$\frac{a+b\sqrt{m}}{2}$ where
   ~$a,b$ are the smallest positive integers satisfying~$a^2 - mb^2=\pm 4$.
    If~$m=5$,~$v_5=\frac{1+\sqrt{5}}{2}$ and if~$m=13$,~$v_{13}=\frac{3+\sqrt{13}}{2}>1+\sqrt{2}$. 
    Otherwise, if ~$m \geq 17$, then ~$v_m \geq \frac{1+\sqrt{m}}{2} \geq \frac{1+\sqrt{17}}{2} > 1+\sqrt{2}$.
\end{proof}

We now turn to the proof of the main theorem.
\begin{proof}[proof of Theorem \ref{multi}]
    Let~$u$ be any element of~$\OLU$. By Lemma \ref{fini},~$2^{n-1}\LOG(u)\in \LOG(E)$.
    Thus if nonzero~$w$ is in~$\bigwedge^2\LOG(\OLU)$, then~$2^{2n-2}w\in \bigwedge^2\LOG(E)$. Thus by \eqref{proto},
    \begin{equation} \label{final}
        \norm{w}{1}=\frac{1}{2^{2n-2}}\norm{2^{2n-2}w}{1}\geq 2\log\left(\frac{1+\sqrt{5}}{2}\right)\log(1+\sqrt{2}) \approx 0.8483.
    \end{equation}
\end{proof}

\begin{remark}
    The lower bound of \eqref{final} may not be optimal. 
    In other words, there might be a greater lower bound for~$\norm{w}{1}$.
    However, for the case ~$L=\Q(\sqrt{2},\sqrt{5})$ and~$w=\pm \LOG(\frac{1+\sqrt{5}}{2}) \wedge \LOG(1+\sqrt{2})$,
    then~$\norm{w}{1}=8\log(\frac{1+\sqrt{5}}{2})\log(1+\sqrt{2}) \approx 3.3930$ and thus the optimal lower bound for \eqref{final} is at most~$3.3930$.
\end{remark}

\end{document}